\documentclass[12pt]{amsart}
\usepackage{amscd}
\usepackage{verbatim}
\usepackage{amssymb}

\DeclareMathAlphabet{\cat}{OT1}{cmss}{m}{sl}

\newtheorem{theorem}{Theorem}[section]

\newtheorem{proposition}[theorem]{Proposition}
\newtheorem{lemma}[theorem]{Lemma}
\newtheorem{corollary}[theorem]{Corollary}

\theoremstyle{definition}
\newtheorem{remark}[theorem]{Remark}

\newtheorem{example}[theorem]{Example}

\newcommand{\gmu}{\boldsymbol{\mu}}

\title[$R$-triviality of ${\rm F}_4$: First Tits Construction] 
{$R$-triviality of Groups of Type ${\rm F}_4$ Arising from the First Tits Construction}

\author
[S. Alsaody]
{Seidon Alsaody}
\thanks{S.\ Alsaody wishes to thank PIMS for their financial support. }

\author
[V. Chernousov]
{Vladimir Chernousov}
\thanks{ V. Chernousov was partially supported by  Alexander von Humboldt Foundation 
and an NSERC research grant} 

\author
[A. Pianzola]
{Arturo Pianzola}
\thanks{A. Pianzola wishes to thank NSERC and CONICET for their
continuous support}

\address
{Department of Mathematical Sciences, University of Alberta,
Edmonton, Alberta, Canada T6G 2G1}

\email{alsaody@ualberta.ca}

\address
{Department of Mathematical Sciences, University of Alberta,
Edmonton, Alberta, Canada T6G 2G1}

\email {vladimir@ualberta.ca}

\address
{Department of Mathematical Sciences, University of Alberta,
Edmonton, Alberta, Canada T6G 2G1}

\email{a.pianzola@ualberta.ca}

\begin{document}

\begin{abstract} Any group of type ${\rm F}_4$ is obtained as the automorphism group of an Albert algebra. 
We prove that such a group is $R$-trivial whenever the Albert algebra is obtained from the first Tits construction. 
Our proof uses cohomological techniques and the corresponding result on the structure group of such Albert algebras.
\end{abstract}

\maketitle

\section{Introduction}
The notion of path connectedness for topological spaces is natural and well-known. In algebraic geometry, this notion corresponds to that of \emph{$R$-triviality}. Recall that the notion of 
$R$-equivalence, introduced by Manin in \cite{M}, is an important
birational invariant of an algebraic variety defined over an arbitrary field  and is defined as 
follows: if ${\bf X}$ is a variety over a field $K$ with ${\bf X}(K)\neq\emptyset$, two $K$-points $x,y\in X(K)$ are called \emph{elementarily $R$-equivalent} if there is a path from $x$ to $y$, i.e.\ if there exists a rational 
map $f:\mathbb A^1\dashrightarrow {\bf X}$ defined at $0,1$ and mapping
0 to $x$ and $1$ to $y$. This generates an equivalence relation $R$ on ${\bf X}(K)$ and one
can consider the set ${\bf X}(K)/R$. 

If ${\bf X}$ has in addition a group structure, i.e. ${\bf X}={\bf G}$ is an algebraic 
group over a field $K$, then the set ${\bf G}(K)/R$ has a natural structure of an (abstract) group and this 
gives rise to a functor 
$$
{\bf G}/R: Fields/K \longrightarrow  Groups,
$$ where $Fields/K$ is the 
category of field extensions of $K$, and ${\bf G}/R$ is given by $L/K \to {\bf G}(L)/R$.  
One says that ${\bf G}$ is \emph{$R$-trivial} if ${\bf G}(L)/R=1$ 
for all field extensions $L/K$.

The computation of the group ${\bf G}(K)/R$ is known only in a few cases. Here are two examples.

\smallskip

\noindent
{\it Example $1$}: Let $A$ be a central simple algebra over a field $K$ and ${\bf G} = {\bf GL}(1,A)$ the
group of ``invertible elements of $A$''. Denote by ${\bf H}$ the kernel ${\bf SL}(1,A)$ of the reduced
norm homomorphism ${\rm Nrd} : {\bf G} \to {\bf G}_m$. 
Then by \cite{V} the group ${\bf G}(K)/R$ equals $A^\times/[A^\times,A^\times]$, 
which is naturally isomorphic to the Whitehead
group ${\rm K}_1(A)$ in algebraic K-theory (cf., \cite[\S\,3]{Mil}). The group ${\bf H}(K)/R$
 is known as the reduced Whitehead group ${\rm SK}_1(A)$. The groups
${\rm K}_1(A)$ and ${\rm SK}_1(A)$ have been extensively studied  in the framework of 
algebraic $K$-Theory. 

\smallskip

\noindent
{\it Example $2$} (cf. \cite{CM1}): let $f$ be a quadratic form over a field $F$ of characteristic not 2
and let ${\bf G}={\bf Spin}(f)$.  Consider the complex
$$
\cdots \longrightarrow \coprod_{x\in {\bf X}_{(1)}} K_2 F(x) \stackrel{\partial}{\longrightarrow} 
\coprod_{x\in {\bf X}_{(0)}} K_1 F(x) \stackrel{N}{\longrightarrow} K_1 F=F^\times
$$
where ${\bf X}$ is the projective quadric corresponding to $f$. Let $A_0({\bf X},K_1)$ be the last 
homology group of this complex. Then ${\bf G}(F)/R\simeq A_0({\bf X},K_1)$.
This result allows one to use the machinery of algebraic $K$-theory when dealing with groups of $R$-equivalence classes and is useful in both directions. For instance,
if $f$ is a Pfister form, then ${\bf G}$ is a rational variety, hence ${\bf G}(F)/R=1$ and therefore 
$A_0({\bf X},K_1)=0$. This 
triviality result is originally due to M.\ Rost and was used in the proof of the Milnor conjecture. 
Conversely, if $f$ is a generic quadratic form over the field $F=k(t_1,\ldots,t_m)$, where
$t_1,\ldots,t_n$ are independent variables over the field $k$, then using techniques of Chow groups
one shows that $A_0({\bf X},K_1)=0$, implying ${\bf G}(F)/R=1$.

\medskip

These two examples show that the group ${\bf G}(K)/R$ encodes deep arithmetic properties
of ${\bf G}$, and that its computation is of great importance in the theory of algebraic groups. 
It is in general a difficult question to determine whether ${\bf G}$ is $R$-trivial. For  algebraic groups, 
this is the case if ${\bf G}$ is \emph{rational}, i.e.\ birationally equivalent to affine space (cf. \cite{CTS}).
Also, it is known that semisimple groups of rank at most 2 are rational, hence $R$-trivial (cf. \cite{V}).

The group ${\bf G}(K)/R$ for algebraic tori was computed in \cite{CTS} and for adjoint
classical semisimple groups in \cite{Merkurjev1}. The latter computation enables one to construct
first examples of non-rational adjoint groups.
For simply connected semisimple groups the computation of the group of
$R$-equivalence classes is known only for  type ${\rm A}_n$ (cf. \cite{CM}), for
spinor groups ${\bf Spin}(f)$ (see the above Example $2$) and in some easy cases when this ${\bf G}/R$ is 
trivial because the corresponding
algebraic group is rational (types ${\rm C}_n$ and ${\rm G}_2$). 

The case of exceptional groups is widely open.
Beyond ${\rm G}_2$, which is rational by
the above, not much is known. In particular it is not known whether groups of type ${\rm F}_4$ are $R$-trivial. These groups arise as automorphism groups of Albert algebras (i.e.\ simple exceptional Jordan algebras, defined below).
Recall that any Albert algebra can be obtained through the first or the second Tits construction.
The following theorem is the main result of our paper.

\begin{theorem}
Any algebraic group of type ${\rm F}_4$ over a field $K$ of characteristic not 2 or 3 arising from the first Tits construction is $R$-trivial.
\end{theorem}

We conclude this introduction by mentioning three standard conjectures about the functor ${\bf G}/R$: (i) it takes values in the category
of abelian groups, (ii) ${\bf G}(K)/R$ is a finite group if $K$ is finitely generated over its prime subfield,
and (iii) the functor ${\bf G}/R$ has transfers. 

If ${\bf G}$ is an algebraic torus, then the first and third conjectures
trivially hold and the second was proved by J.-L.\  Colliot-Th\'el\`ene and J.-J.\ Sansuc \cite{CTS}.
For semisimple groups not much is known. Our main result provides an indication that all three conjectures 
are true for reductive algebraic groups.

\medskip

\noindent
{\it Acknowledgements}: 
We wish to thank Holger Petersson for his answers to our questions.

\section{Preliminaries on algebraic groups}

In this section we collect some facts about algebraic groups for later use. We start with an observation on subgroups of type ${\rm D}_4$ inside the split group of type ${\rm F}_4$.

\begin{proposition}\label{splitD_4} Let ${\bf H}$ be a split group of type ${\rm F}_4$ defined over an arbitrary field $K$.
Then any $K$-subgroup ${\bf M}\subset {\bf H}$  of type type ${\rm D}_4$ is quasi-split.
\end{proposition}
\begin{proof} Consider a Borel $K$-subgroup ${\bf B}$ of $\bf H$. Since the dimension of ${\bf B}$ is 28, it intersects the 28-dimensional subgroup ${\bf M}$ non-trivially.
Thus ${\bf M}$ contains a split unipotent subgroup or 
a split torus ${\bf G}_m$. In both cases, ${\bf M}$ is $K$-isotropic. 

Let ${\bf S}\subset {\bf M}$ be a maximal $K$-split torus. If ${\rm dim}({\bf S})\geq 3$, then ${\bf M}$ is necessarily quasi-split. It remains
to check the cases where the dimension of ${\bf S}$ is 1 or 2. Note that if ${\bf M}$ is a trialitarian group of type $^{3,6}{\rm D}_4$
and if ${\rm dim}({\bf S})=2$ then ${\bf M}$ is quasi-split. The remaining cases correspond to the following Tits $K$-indices
(cf. \cite{T66}).

\subsection{Type $^1{\rm D}_4$}
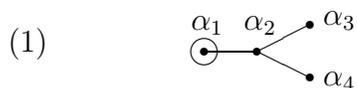
\begin{equation}
\label{1}
\begin{picture}(250,20)
\put(00,00){\line(1,0){20}}
\put(20,0){\line(2,-1){20}}
\put(20,0){\line(2,1){20}}
\put(0,0){\circle*{3}}
\put(20,0){\circle*{3}}
\put(40,10){\circle*{3}}
\put(40,-10){\circle*{3}}
\put(00,0){\circle{10}}
\put(-5,8){$\alpha_1$}
\put(15,8){$\alpha_2$}
\put(45,10){$\alpha_{3}$}
\put(45,-13){$\alpha_{4}$}
\end{picture}       
\end{equation}
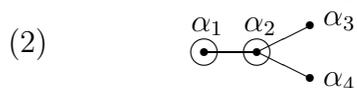
\begin{equation}
\label{2}
\begin{picture}(250,20)
\put(00,00){\line(1,0){20}}
\put(20,0){\line(2,-1){20}}
\put(20,0){\line(2,1){20}}
\put(0,0){\circle*{3}}
\put(20,0){\circle*{3}}
\put(40,10){\circle*{3}}
\put(40,-10){\circle*{3}}
\put(00,0){\circle{10}}
\put(20,0){\circle{10}}
\put(-5,8){$\alpha_1$}
\put(15,8){$\alpha_2$}
\put(45,10){$\alpha_{3}$}
\put(45,-13){$\alpha_{4}$}
\end{picture}       
\end{equation}

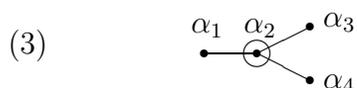
\begin{equation}
\label{3}
\begin{picture}(250,20)
\put(00,00){\line(1,0){20}}
\put(20,0){\line(2,-1){20}}
\put(20,0){\line(2,1){20}}
\put(0,0){\circle*{3}}
\put(20,0){\circle*{3}}
\put(40,10){\circle*{3}}
\put(40,-10){\circle*{3}}
\put(20,0){\circle{10}}
\put(-5,8){$\alpha_1$}
\put(15,8){$\alpha_2$}
\put(45,10){$\alpha_{3}$}
\put(45,-13){$\alpha_{4}$}
\end{picture}       
\end{equation}

\medskip

\subsection{Type $^2{\rm D}_4$}

\medskip

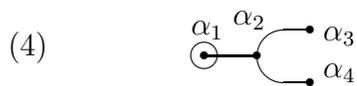
\begin{equation}
\label{4}
\begin{picture}(250,20)
\put(00,00){\line(1,0){20}}
\put(40,0){\oval(40,20)[l]}
\put(0,0){\circle*{3}}
\put(20,0){\circle*{3}}
\put(40,10){\circle*{3}}
\put(40,-10){\circle*{3}}
\put(0,0){\circle{10}}
\put(-5,8){$\alpha_1$}
\put(11,12){$\alpha_2$}
\put(45,6){$\alpha_{3}$}
\put(45,-8){$\alpha_{4}$}
\end{picture}       
\end{equation}

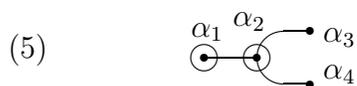
\begin{equation}
\label{5}
\begin{picture}(250,20)
\put(00,00){\line(1,0){20}}
\put(40,0){\oval(40,20)[l]}
\put(0,0){\circle*{3}}
\put(20,0){\circle*{3}}
\put(40,10){\circle*{3}}
\put(40,-10){\circle*{3}}
\put(0,0){\circle{10}}
\put(20,0){\circle{10}}
\put(-5,8){$\alpha_1$}
\put(11,12){$\alpha_2$}
\put(45,6){$\alpha_{3}$}
\put(45,-8){$\alpha_{4}$}
\end{picture}       
\end{equation}

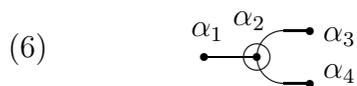
\begin{equation}
\label{6}
\begin{picture}(250,20)
\put(00,00){\line(1,0){20}}
\put(40,0){\oval(40,20)[l]}
\put(0,0){\circle*{3}}
\put(20,0){\circle*{3}}
\put(40,10){\circle*{3}}
\put(40,-10){\circle*{3}}
\put(20,0){\circle{10}}
\put(-5,8){$\alpha_1$}
\put(11,12){$\alpha_2$}
\put(45,6){$\alpha_{3}$}
\put(45,-8){$\alpha_{4}$}
\end{picture}       
\end{equation}

\bigskip

\subsection{Type $^{3,6}{\rm D}_4$}

\medskip

\begin{equation}
\label{7}
\begin{picture}(250,20)
\put(00,00){\line(1,0){30}}
\put(30,0){\oval(60,20)[l]}
\put(0,0){\circle*{3}}
\put(30,0){\circle*{3}}
\put(30,10){\circle*{3}}
\put(30,-10){\circle*{3}}
\put(00,0){\circle{10}}
\end{picture}
\end{equation}

\bigskip

We will show that neither case can occur.

\smallskip

\noindent
{\it Case} (\ref{1}): Here ${\rm dim}({\bf S})=1$. The semisimple anisotropic kernel ${\bf M}_{ss}$ of ${\bf M}$ is a group
of type ${^1}A_3$, and is thus of  dimension $15$. From the theory of algebraic groups, the centralizer $C_{{\bf H}}({\bf S})$ in $\bf H$ is a $K$-split\footnote{It is $K$-split because every split subtorus in ${\bf H}$, 
in particular ${\bf S}$, is contained in a maximal split torus of ${\bf H}$.} 
reductive group of rank $4$ and by construction contains ${\bf M}_{ss}$. Hence the semisimple part
$$C_{\bf H}({\bf S})_{ss}=[C_{{\bf H}}({\bf S}),C_{{\bf H}}({\bf S})]$$ 
of $C_{{\bf H}}({\bf S})$ 
is a semisimple $K$-split group of rank $3$ containing ${\bf M}_{ss}$. Since $C_{\bf H}({\bf S})_{ss}$ is split, $C_{\bf H}({\bf S})_{ss}
\not={\bf M}_{ss}$. Therefore the type of $C_{\bf H}({\bf S})_{ss}$ is ${\rm B}_3$ or ${\rm C}_3$, 
and so its dimension is $21$.  
In both cases, any Borel $K$-subgroup of $C_{\bf H}({\bf S})_{ss}$ has dimension 
$12$, and thus for dimension reasons intersects ${\bf M}_{ss}$ nontrivially, contradicting the assumption that ${\bf M}_{ss}$ is $K$-anisotropic.

\smallskip

\noindent
{\it Case} (\ref{2}): Here ${\rm dim}({\bf S})=2$, and the semisimple anisotropic kernel ${\bf M}_{ss}$ of ${\bf M}$ is a group
of type $A_1\oplus A_1$, hence of dimension $6$. Again, we consider the centraliser $C_{{\bf H}}({\bf S})$;
its semisimple part $C_{\bf H}({\bf S})_{ss}=[C_{{\bf H}}({\bf S}),C_{{\bf H}}({\bf S})]$ is a $K$-split semisimple group of rank $2$ containing ${\bf M}_{ss}$. Therefore it is either of
type ${\rm A}_2$ or
${\rm B}_2={\rm C}_2$. A dimension count shows that the intersection of a Borel $K$-subgroup of
$C_{\bf H}({\bf S})_{ss}$ and ${\bf M}_{ss}$ is nontrivial, contradicting anisotropy.

\smallskip

\noindent
{\it Case} (\ref{3}): Here, up to isogeny, ${\bf M}$ is of the form ${\bf O}^+(D,f)$ where 
$D$ is a quaternion algebra and $f$ is a skew-hermitian form over $D$  of dimension $4$ 
with Witt index  $1$.
It is known that after passing to the function field of the Severi--Brauer variety of $D$, 
 the anisotropic part of $f$
remains anisotropic, hence
the Tits index of ${\bf M}$ is of the form (\ref{2}). But this case is impossible by the above consideration.

Cases (4), (5) and (6) are treated in the same way as cases (1), (2) and (3) respectively, since the dimensions of the corresponding subgroups are the same. 
In case (7), the uncircled vertices correspond to the Weil restriction $R_{L/K}({\rm SL}(1,D))$ for a cubic field extension $L$ of $K$ and a quaternion $L$-algebra $D$. Extending scalars
to $L$ reduces this case to one of the previously treated cases (3) and (6). This completes the proof.
\end{proof}

\subsection{The norm principle}

Let ${\bf G}$ be a semisimple group defined over a field $K$ and ${\bf Z}\subset {\bf G}$ a central subgroup. Let moreover $c=[\xi]$ be an element in
\begin{equation}\label{R-trivial}
{\rm Ker}\,[H^1(K,{\bf Z})\longrightarrow H^1(K,{\bf G})].
\end{equation}

\smallskip

\noindent
{\bf Definition.} We say that $[\xi]$ is \emph{$R$-trivial} if there exists
$$
[\xi(t)]\in {\rm Ker}\,[H^1(K(t),{\bf Z})\longrightarrow H^1(K(t),{\bf G})],
$$
where $K(t)$ is a purely transcendental extension of $K$, such that $\xi(t)$ is defined at $t=0$ and $t=1$, and satisfies $[\xi(0)]=1$ and
$[\xi(1)]=[\xi]$.

\begin{remark} Here, if ${\bf J}$ is an algebraic group over $K$, we say that a cocycle in $Z^1(K(t),{\bf J})$) is \emph{defined at 0 and 1} if it lies in the image of
$Z^1(\mathcal O,{\bf J})$, where $\mathcal O$ is the intersection of the localizations $K[t]_{(t)}$ and $K[t]_{(t-1)}$ in $K(t)$.
Using the evaluation maps $\varepsilon_0,\varepsilon_1:\mathcal O\to K$, such a cocycle can be evaluated at 0 and 1.
\end{remark}

\begin{remark}  
The short exact sequence
\[1\longrightarrow {\bf Z}\longrightarrow {\bf G}\longrightarrow {\bf G}/{\bf Z}\longrightarrow 1\]
induces an exact sequence
\[1\longrightarrow {\bf Z}(K)\longrightarrow {\bf G}(K)\longrightarrow ({\bf G}/{\bf Z})(K)\overset{f}{\longrightarrow} H^1(K,{\bf Z})\longrightarrow H^1(K,{\bf G})\]
and the map $f$ induces a bijection between the
orbits of ${\bf G}(K)$ in $({\bf G}/{\bf Z})(K)$ and the kernel in (\ref{R-trivial}).
\end{remark}

\smallskip

\begin{example}\label{quasi-split}
 Let ${\bf G}$ be a quasi-split (absolutely) simple simply connected group. Then 
$$
H^1(K,{\bf Z})= {\rm Ker}\,[H^1(K,{\bf Z})\to H^1(K,{\bf G})] 
$$
and every $[\xi]\in H^1(K,{\bf Z})$ 
is $R$-trivial.
Indeed, ${\bf Z}$ is contained in a maximal $K$-quasi-split torus $T$ and $H^1(F,T)=1$ for any field extension $F/K$.
It remains to note that the variety ${\bf G}/{\bf Z}$ is $K$-rational, which implies that
all elements in $H^1(K,{\bf Z})$ are $R$-trivial.
\end{example}

\begin{example}\label{typeA_n} 
Let ${\bf G}={\bf SL}(1,D)$ where $D$ is a central simple algebra of arbitrary 
degree $n$ over $K$, and let ${\bf Z}$ be the centre of ${\bf G}$. Since
${\bf Z}\simeq \gmu_n$, we have $H^1(K,{\bf Z})\simeq K^\times/K^{\times n}$. Also, $H^1(K,{\bf G})\simeq K^\times /{\rm Nrd}\,(D^\times)$.
Therefore,
$$
{\rm Ker}\,[H^1(K,{\bf Z})\longrightarrow H^1(K,{\bf G})]={\rm Nrd}(D^\times)/K^{\times n}.
$$
Since $D$ is an affine space, any $[\xi]$ in the above kernel is $R$-trivial. 
\end{example}
\smallskip

For any finite  extension $L/K$ we have the restriction map
$$
res^L_K: H^1(K,{\bf Z}) \to H^1(L,{\bf Z})
$$
and the corestriction map 
$$
cor^L_K: H^1(L,{\bf Z})\to 
H^1(K,{\bf Z}).
$$ 
\noindent
{\bf Definition.} Let $L/K$ be a finite field extension. 
We say that \emph{the norm principle holds for a cohomology class}
$$
[\eta]\in {\rm Ker}\,[H^1(L, {\bf Z})\longrightarrow H^1(L,{\bf G})]
$$
if
$$
cor^L_K([\eta]) \in {\rm Ker}\,[H^1(K,{\bf Z}) \longrightarrow H^1(K,{\bf G})].  
$$
We also say that the norm principle holds for the pair $({\bf Z},{\bf G})$ if it holds for all
classes
\[ [\eta]\in {\rm Ker}\,[H^1(L, {\bf Z})\longrightarrow H^1(L,{\bf G})]\]
for each finite field extension $L/K$.

\begin{theorem}\label{normprinciple}{\rm (P.\ Gille \cite{Gille})} Let $L/K$ be a finite field extension and assume that
$[\eta]\in {\rm Ker}\,[H^1(L, {\bf Z})\longrightarrow H^1(L,{\bf G})]$
is $R$-trivial. Then the norm principle holds for $[\eta]$.
\end{theorem}

\section{Preliminaries on Albert algebras}\label{prel}

\subsection{Albert Algebras}
Let $K$ be a field of characteristic different from 2 and 3. A \emph{Jordan algebra} over $K$ is a unital, commutative $K$-algebra\footnote{We do not assume the algebra to be associative.
} $A$ in which the Jordan identity
\[(xy)(xx)=x(y(xx))\]
holds for all $x,y\in A$ (thus $A$ is power associative). 

If $B$ is an associative algebra with multiplication denoted by $\cdot$, then the anticommutator $\frac{1}{2}(x\cdot y+y\cdot x)$ 
 endows $B$ with a Jordan algebra structure,
which we denote by $B^+$. A Jordan algebra $A$ is called \emph{special} if it is isomorphic to a Jordan 
subalgebra of $B^+$ for some associative algebra $B$, and \emph{exceptional} otherwise.
An \emph{Albert algebra} is then defined as a simple, exceptional Jordan algebra. 
It is known that the dimension of any Albert algebra is 27, and that over separably closed fields,
all Albert algebras are isomorphic \cite[37.11]{KMRT}. From this it follows that all Albert algebras over $K$ are twisted forms 
of each other.

Over a general  field $K$, any Albert algebra arises from one of the two \emph{Tits constructions}. In this note, we will mainly be concerned with those arising from the first Tits
construction. Recall that in this case, $A=D\oplus D\oplus D$ as a vector space, where $D$ is a central simple algebra over
$K$ of degree $3$. The multiplication of $A$ is determined by a scalar $\mu\in K^*$. More precisely, define the cross product on $D$ by
\[u\times v=uv+vu-T_D(u)v-T_D(v)u-T_D(uv)+T_D(u)T_D(v),\]
where juxtaposition denotes the associative multiplication of $D$, and $T_D$ denotes the reduced trace, and write
\[\widetilde u=T_D(u)-u\]
for $u\in D$. For any  $\mu\in K^\times$, the product $(x,y,z)(x',y',z')$ in $A=D\oplus D\oplus D$ is then given by
\[\frac{1}{2}(xx'+x'x+\widetilde{yz'}+\widetilde{y'z},\widetilde{x}y'+\widetilde{x'}y+\mu^{-1}z\times z',z\widetilde{x'}+z'\widetilde{x}+\mu y\times y').\] 
One then writes $A=J(D,\mu)$ and says that $A$ arises from $D$ and $\mu$ via the first Tits construction.

In some arguments, we will also need to consider reduced Albert algebras, and we briefly recall their construction. Given an octonion algebra $C$ over $K$ and 
$\Gamma=(\gamma_1,\gamma_2,\gamma_3)\in (K^\times)^3$, the set $H_3(C,\Gamma)$ of $\Gamma$-hermitian $3\times 3$-matrices over $C$, i.e.\ those of the form
\[\left(\begin{array}{ccc}
   \xi_1 & c_3 & \gamma_1^{-1}\gamma_3\overline{c_2}\\
   \gamma_2^{-1}\gamma_1\overline{c_3} & \xi_2 & c_1\\
   c_2 & \gamma_3^{-1}\gamma_2\overline{c_1} & \xi_3
  \end{array}\right),
\]
with $\xi_i\in K$ and $c_i\in C$ for $i=1,2,3$, and where $\overline{\phantom{c}}$ denotes the involution on $C$, is an Albert algebra under the anticommutator of the matrix product.
From \cite{Sch} we know that an Albert algebra over $K$ is a division algebra if and only if it is not of this form.

\subsection{Automorphisms and Similarities}
If $A$ is an Albert algebra over $K$, then ${\bf H}={\bf Aut}(A)$ is a simple algebraic group  over $K$ 
of type ${\rm F}_4$. In fact, the assignment $A\mapsto {\bf Aut}(A)$ establishes an equivalence between
the category of Albert algebras over $K$ and the category of simple { $K$-groups} of type ${\rm F}_4$. Moreover, $A$ is endowed with a cubic form $N:A\to K$ known as the \emph{norm} of $A$. 
The \emph{structure group} ${\bf Str}(A)$ of $A$ is the affine group scheme  whose $R$-points are 
defined by
$$
{\bf Str}(A)(R)=\{ x\in {\bf GL}(A)(R)\ | \ x\ \text{is an isotopy of}\ A\otimes R ,\}
$$  
for every $K$-ring $R$. By an \emph{isotopy} of an Albert algebra $B$ over a ring $R$ 
one understands 
an isomorphism of Jordan algebras from $B$ to the isotope $B^{(p)}$ for some 
$p\in B^\times$   with $N(p)\in R^\times$, the
algebra $B^{(p)}$ having the same underlying module as $B$, with multiplication 
\[x\cdot_p y=x(yp)+y(xp)-(xy)p.\]

It is clear that every isotopy is a norm similarity  of $A$. This holds for Albert algebras 
over commutative rings (see \cite{P} for a discussion of such algebras) and implies that the
structure group is a subgroup of the group of norm similarities. 
By \cite[VI Theorem 7]{Jac},
if $F$ is a field with more than 3 elements, then every norm similarity is an isotopy. Since these groups are smooth, it follows that they coincide, i.e.\ for any $K$-ring $R$,
\[{\bf Str}(A)(R)=\{ x\in {\bf GL}(A)(R)\ | \ N_R(x(a))=\nu(x) N_R(a)\ \ \forall \,a\in A\otimes R\ \},\]
where $N_R$ is the scalar extension of $N$ to $A\otimes R$, and $\nu(x)$ is a 
scalar  in $R^\times$ 
(called a multiplier) depending on $x$ only. Note that ${\bf Str}(A)$ contains a (central) split 
torus ${\bf G}_m$ consisting of all homotheties of $A$. 

It is well known that the derived subgroup $${\bf G}=[{\bf Str}(A),{\bf Str}(A)]$$ of ${\bf Str}(A)$ 
is a strongly inner form of a split simple simply connected algebraic group
of type ${\rm E}_6$ and that ${\bf Str}(A)$ is an almost direct product of ${\bf G}_m$ and ${\bf G}$ (their intersection is the centre 
of $\bf{G}$). Thus, $$\overline{{\bf G}}={\bf Str}(A)/{\bf G}_m$$ is an adjoint group of type ${\rm E}_6$.
Since our main result is known for split or isotropic ${\bf G}$ we may assume without loss of generality 
that ${\bf G}$ is $K$-anisotropic which is equivalent  to saying that $A$ is a division algebra or equivalently 
the norm map $N$ is anisotropic, i.e.\
the equation $N(a)=0$ has no non-zero solutions over $K$.

\subsection{Action of the structure group ${\bf Str}(A)$ on $A$}
We denote the group of $K$-points of ${\bf Str}(A)$ (resp. ${\bf G},\, {\overline{\bf G}},\, {\bf H}$) by
${\rm Str}(A)$ (resp. $G,\,\overline{G},\,H$). Under the natural action of ${\bf Str}(A)$ on $A$, the group ${\bf H}$ coincides with the stabilizer of  $1\in A$, as follows from
\cite[5.9.4]{SV}.
This action of ${\bf Str}(A)$ on $A$ induces an action 
on the open subset $U$ defined by
$$
U(R)=\{ a\in A\otimes_KR\ |\ N_R(a)\in R^\times\}.
$$
for any $K$-ring $R$.
\begin{lemma} 
Let $A$ be split over $K$. Then 
 the action ${\rm Str}(A)\times U \to U$
is transitive.
\end{lemma}

It is known that the action is transitive over a separable closure of $K$. Therefore, to prove the lemma,
it suffices to show that the natural map $H^1(K,{\bf H})\to H^1(K,{\bf Str}(A))$ has trivial kernel
or equivalently 
$$
{\rm Ker}\,[H^1(K,{\bf H})\longrightarrow H^1(K,{\bf G})]=1.
$$

We will give two proofs of this statement, one using the structure theory of algebraic groups, and the other more Jordan theoretic in nature, to illuminate both sides of the statement.

\begin{proof}[First proof] Let $[\xi]$ be in the kernel. Consider the twisted group ${^\xi{\bf H}}$. 
It is a $K$-subgroup
of the split group ${^\xi{\bf G}}\simeq {\bf G}$. Let ${\bf B}\subset {\bf G}$ be a Borel $K$-subgroup.
A dimension count shows that ${^\xi{\bf H}}\cap {\bf B}$ is nontrivial. Hence ${^\xi{\bf H}}$ is isotropic, which
implies that the invariants $g_3$ and $f_5$ of ${^\xi{\bf H}}$ are both trivial. Thus, up to equivalence in 
$H^1(K,{\bf H})$,
we may assume that $\xi$ takes values in a split subgroup of ${\bf H}$ of type ${\rm G}_2$ which is in turn a subgroup of 
the split group ${\bf M}$ of type $\mathrm{D}_4$ generated by the roots $\alpha_2,\ldots,\alpha_5$.
Since $[\xi]$ is trivial in $H^1(K,{\bf G})$, it follows  from the Tits classification of isotropic groups of type $^1\mathrm{E}_6$
that it is trivial in $H^1(K,{\bf M})$, hence 
in $H^1(K,{\bf H})$. 
\end{proof}

\begin{proof}[Second proof]
By \cite[Theorem 2.5]{Als}, this kernel parametrises the isomorphism classes of isotopes of $A$. But from \cite[Corollary 60]{P},
it follows that the invariants mod 2 and 3 of any isotope of the split Albert algebra $A$ are trivial. Thus all such isotopes are isomorphic, and the above kernel
is trivial, as desired.
\end{proof}

\subsection{Subgroups and a cohomology class attached to a point in $A$}

Let $a\in A$. We  associate to it a subgroup of type ${\rm D}_4$ in ${\bf G}$ and a $2$-dimensional torus
as follows.
Let $L\subset A$ be a subalgebra generated by $a$. It is known that $L$ is a  subfield 
since $A$ is a division algebra. Define the group ${\bf G}^L$ by
$$
{\bf G}^L(R)=\{ x\in {\bf G}(R) \ | \ x(l)=l \ \ \forall l\in L\otimes_KR\}
$$
for any $K$-ring $R$. Since ${\bf G}^L$ stabilizes $1\in L\subset A$ one gets ${\bf G}^L\subset {\bf H}\subset {\bf G}$.
It is known that over a separable closure of $K$ the group ${\bf G}^L$ is conjugate
to the standard subgroup in ${\bf G}$ of type ${\rm D}_4$ generated by roots $\alpha_2,\alpha_3,\alpha_4,
\alpha_5$. Therefore ${\bf S}'^L=C_{\bf G}({\bf G}^L)\subset {\bf G}$ is a $2$-dimensional torus over $K$. 
Using an explicit model of $A$ over a separable closure of $K$ 
one can easily verify that $${\bf Z}^L:={\bf S}'^L\cap {\bf H}={\bf S}'^L\cap 
{\bf G}^L$$ is the centre of ${\bf G}^L$. 

Let ${\bf S}^L\subset {\bf Str}(A)$ be a $3$-dimensional torus  generated by ${\bf S}'^L$ and ${\bf G}_m$.
Note that ${\bf S}^L\cdot {\bf G}^L$ is the connected component
of the stabilizer ${\bf N}^L={\bf Str}(A,L)$ defined by
$$
{\bf N}^L(R)={\bf Str}(A,L)(R)=\{ x\in {\bf Str}(A)(R) \ | \ x(L\otimes_KR)=L\otimes_KR\}
$$
for each $K$-ring $R$, and that over the separable closure of $K$ one has
${\bf N}^L/({\bf S}^L\cdot {\bf G}^L)\simeq S_3.$  Since the representatives 
of this quotient group can be chosen in ${\bf H}$, we have a natural identification
${\bf N}^L/({\bf S}^L\cdot {\bf G}^L)\simeq {\bf Aut}(L/K)$.

Our next goal is to describe a structure of ${\bf S}^L$. Before doing so,
we recall a standard fact about
actions of algebraic groups on varieties and the corresponding transporters. 

Let $L_1,L_2\subset A$ be two cubic \'etale (commutative) subalgebras. Consider
the transporter defined by
$$
{\bf Transp}(L_1,L_2)(R)=\{ x\in {\bf Str}(A)(R)\ | \ x(L_1\otimes_KR)=L_2\otimes_KR\}.
$$
for each $K$-ring $R$. It is an ${\bf N}^{L_1}$-torsor, hence it is represented by a cohomology
class $$
[\xi]\in {\rm Ker}[H^1(K,{\bf N}^{L_1})\longrightarrow H^1(K,{\bf Str}(A)].$$ 
By construction, this class is trivial if and only if $L_1$ and $L_2$
are conjugate over $K$. Furthermore, let $[\bar{\xi}]$ be the image of 
$[\xi]$ in 
$$H^1(K,{\bf N}^{L_1}/ ({\bf S}^{L_1}\cdot {\bf G}^{L_1}))= 
H^1(K,{\bf Aut}(L_1/K)).
$$
Then we have ${^{\bar{\xi}}L}_1\simeq L_2$. In particular, if $L_1\simeq L_2$,
then $[\bar{\xi}]$ is trivial, hence up to equivalence  in $H^1(K,{\bf N}^{L_1})$ 
we may assume that  $\xi$ takes values in 
${\bf S}^{L_1}\cdot {\bf G}^{L_1}$. This observation leads us to the following.
\begin{lemma}
Let $A=H_3(C,\Gamma)$  be a reduced Albert algebra, where 
$C$ is an octonion algebra over $K$ and $\Gamma\in (K^\times)^3$.
Let $L_1\subset A$ be the diagonal subalgebra and let $L_2\subset A$ be an arbitrary
 split cubic subalgebra. Then $L_1$ and $L_2$ are conjugate in ${\rm Str}(A)$.
\end{lemma}
\begin{proof} By the discussion preceding the lemma, the  obstacle for conjugacy is a cohomology class
$[\xi]\in H^1(K, {\bf S}^{L_1}\cdot {\bf G}^{L_1})$ which maps to the trivial class in
 $H^1(K,{\bf Str}(A))$. Since
$L_1$ consists of diagonal matrices in $H_3(C,\Gamma)$, it readily follows that ${\bf S}^{L_1}$
is a split $K$-torus, hence ${\bf S}^{L_1}\cdot {\bf G}^{L_1}$ 
is a Levi subgroup of a parabolic subgroup ${\bf P}\subset {\bf Str}(A)$. It is well known that
both arrows 
$$
H^1(K,{\bf S}^{L_1}\cdot {\bf G}^{L_1}) \longrightarrow H^1(K,{\bf P})
\longrightarrow H^1(K,{\bf Str}(A))
$$
have trivial kernel. It follows that $[\xi]=1$, as desired.
\end{proof}

We come back to an arbitrary Albert algebra $A$ over $K$, and a cubic subfield $K\subset L\subset A$.

\begin{proposition}  
One has ${\bf S}^L\simeq R_{L/K}({\bf G}_m)$ and ${\bf S}'^L\simeq R_{L/K}^{(1)}({\bf G}_m)$.
\end{proposition}
\begin{proof} 
Let $F/K$ (resp. $F'/K$) be the minimal splitting field of $ R_{L/K}^{(1)}({\bf G}_m)$ 
(resp. ${\bf S}'^L$).
Let $\Gamma={\rm Gal}(F/K)$ and $\Gamma'={\rm Gal}(F'/K)$. These two Galois groups act naturally on
the corresponding character lattices of the tori $R_{L/K}^{(1)}({\bf G}_m)$ and
${\bf S}'^L$, respectively, and hence both embed naturally into ${\rm GL}_2(\mathbb{Z})$.
Recall that an arbitrary $n$-dimensional torus is determined uniquely by its minimal splitting field and 
the conjugacy class of the image of  the corresponding Galois group in ${\rm GL}_n(\mathbb{Z})$. 

\smallskip

\noindent
{\it Step $1$}: $F'\subset F$. The base extension $F/K$ completely splits $L$. Hence the $g_3$-invariant of $A$ becomes trivial over $F$,
implying that $A_F\simeq H_3(C,\Gamma)$  where $C$ is the 
octonion algebra over $F$ corresponding
to $f_3(A_F)$, and $\Gamma\in (F^\times)^3$. By the above lemma, up to conjugacy we may assume that $L\otimes_KF\subset A_F$ coincides with the diagonal subalgebra.
Explicit computations show that $({\bf S}^L)_F$ is a split $3$-dimensional torus, and that so is $({\bf S}'^L)_F$. 

\smallskip

\noindent
{\it Step $2$}: $F\subset F'$.  Since the field extension $F'/K$ splits ${\bf S}^L$ it follows 
from Tits' classification of isotropic simple algebraic groups of type $^1\mathrm{E}_6$
that the semisimple part ${\bf G}^L $ of 
$C_{{\bf Str}(A)}({\bf S}^L)$ is isomorphic to ${\bf Spin}(f)$ where
$f$ is the $3$-fold Pfister form as above, which is either split over $F'$
or splits over a quadratic field extension
$E/F'$. Therefore, the same holds for ${\bf Str}(A)$ and for $A$.
In both cases the $g_3$-invariant of $A_{F'}$, being an invariant mod 3, is trivial,
implying that $A_{F'}\simeq H_3(C,\Gamma)$ for an octonion $F'$-algebra $C$ and $\Gamma\in (F'^\times)^3$. Since all maximal split tori in ${\bf Str}(A_{F'})$ are conjugate 
over $F'$, we may assume that, up to conjugacy, the split torus $({\bf S}^L)_{F'}$ preserves the diagonal subalgebra of $H_3(C,\Gamma)$;
then so does its centralizer in ${\bf Str}(A_{F'})$. Furthermore, explicit computations show that
$$
({\bf G}^L)_{F'}=[C_{\bf{Str}(A_{F'})}({\bf S}^L)_{F'}), C_{\bf{Str}(A_{F'})}({\bf S}^L)_{F'})]
$$ 
acts trivially on the diagonal. 
Since $L$ consists of precisely those points of $A$ fixed pointwise under the action of ${\bf G}^L$, it follows that $L\otimes_K F'\subset A_{F'}$
coincides with the diagonal. Thus $L$ is split over $F'$. 

\smallskip

\noindent
{\it Step $3$}: {\it the images of $\Gamma$ and $\Gamma'$ in ${\rm GL}_2(\mathbb{Z})$ are conjugate}.
Recall that up to conjugacy the group ${\rm GL}_2(\mathbb{Z})$ has two maximal finite subgroups, isomorphic to 
the dihedral groups ${\rm D}_4$ and $D_6\simeq D_3\times \mathbb{Z}/2$. Since the orders of $\Gamma$
and $\Gamma'$ are divisible by $3$ we may, up to conjugacy, assume that
$\Gamma,\Gamma'\subset D_6\subset {\rm GL}_2(\mathbb{Z})$. If $|\Gamma|=|\Gamma'|=3$, then these groups coincide
with the unique $3$-Sylow subgroup of $D_6$ and we are done. Otherwise 
$\Gamma$ and $\Gamma'$ have order $6$, hence
$\Gamma\simeq \Gamma'\simeq S_3$.
The group $D_6\simeq D_3\times \mathbb{Z}/2\simeq S_3\times \mathbb{Z}/2$ has two subgroups isomorphic
to $S_3$. They are not conjugate in $D_6$, but explicit computations shows that they are conjugate in ${\rm GL}_2(\mathbb{Z})$. This completes the proof.
\end{proof}

In the proof we used the fact that the inclusion
$$
L\subset \{ a\in A \ | \ g(a)=a\ \ \forall g\in {\bf G}^L(K)\}
$$
is an equality.
Hence the restriction of the action of ${\bf Str}(A)$ (on $A$) at ${\bf S}^L$ defines an action
${\bf S}^L\times L^\times \to L^\times$. Since ${\bf S}^L\cap {\bf H}=
{\bf S}'^L\cap {\bf H}={\bf Z}^L$, the 
stabilizer of $1\in L^\times$ in ${\bf S}^L$ is 
equal to ${\bf Z}^L$.
It follows that over a separable closure of $K$  the action ${\bf S}^L\times L^\times\to L^\times$ is transitive.
Thus we may identify the varieties ${\bf S}^L/{\bf Z}^L\simeq L^\times$. The exact sequence
$$
1\longrightarrow {\bf Z}^L \longrightarrow {\bf S}^L \overset{\phi}{\longrightarrow} L^\times \longrightarrow 1
$$
and the fact that $H^1(K,{\bf S}^L)=1$ imply that we have a canonical surjective map
$L^\times \to H^1(K,{\bf Z}^L)$. 
Thus to each point $a\in L^\times$ we can attach its image 
 $[\xi_a]\in H^1(K,{\bf Z}^L)$. 
 
\begin{remark}\label{remark} 
By construction, the map $\phi$ is the restriction of the map \linebreak ${\bf Str}(A)\to U$ 
sending $g$ to $g(1)$. 
\end{remark}
 
\section{$R$-triviality}

We keep the notation introduced in Section \ref{prel}; in particular $U\subset A$ is the open subset 
consisting
of the elements $a$ in $A$ whose norm $N(a)$ is invertible.

\begin{proposition}\label{transitivity} 
 Let $A$ be an Albert algebra arising from the first Tits construction.
Then the group ${\rm Str}(A)$ acts  transitively on $U$.
\end{proposition}
\begin{proof} Recall that the stabilizer of $1\in A$ is ${\bf H}$. Thus we may identify
the variety ${\bf G}/{\bf H}$ with $U$. It follows from the exact sequence
$$1\longrightarrow {\bf H} \longrightarrow {\bf Str}(A) \longrightarrow U \longrightarrow 1
$$
that  
there is a canonical one-to-one correspondence between the set
of ${\rm Str}(A)$-orbits in $U$ and the set
$$
{\rm Ker}\,[H^1(K,{\bf H}) \longrightarrow H^1(K, {\bf Str}(A)].
$$
Therefore  our assertion amounts to proving triviality of the above kernel. 

Let $a\in U$ and let 
$$
[\eta]\in {\rm Ker}\,[H^1(K,{\bf H})\longrightarrow H^1(K,{\bf Str}(A))]$$ be the corresponding cohomology class. We take a cubic subfield $L$ containing $a$,
 construct the class $[\xi_a]\in H^1(K,{\bf Z}^L)$ as in the end of the previous section, and write $[\xi]_{{\bf H}}$ for
its image under the composition  
$$H^1(K,{\bf Z}^L)\to H^1(K,{\bf G}^L)\to H^1(K,{\bf H}).$$ By Remark \ref{remark}, $[\eta]=[\xi]_{{\bf H}}$.
The following claim completes the proof of the proposition.

\medskip

\noindent
{\it Claim}: {\it the map $H^1(K,{\bf Z}^L)\to H^1(K,{\bf H})$ is trivial}.

\medskip

\noindent
Indeed, let $[\xi]\in H^1(K,{\bf Z}^L)$.
Take its restriction $[\xi]_L\in H^1(L,(({\bf Z}^L)_L)$.
Upon extending scalars to $L$, the $g_3$-invariant of ${\bf H}$ becomes trivial. Since $A$ is obtained from the first Tits construction, its invariants mod 2 are trivial, and
thus the group ${\bf H}_L$ is split. Then
by Lemma~\ref{splitD_4}, the group ${\bf G}^L$ is quasi-split over $L$. Choose a maximal 
quasi-split torus
$T_L$ in ${\bf G}^L\otimes_K L$. It contains the centre $({\bf Z}^L)_L$ 
of $({\bf G}^L)_L$ and has trivial cohomology in dimension $1$. This implies that
$$
[\xi]_L\in {\rm Ker}\,[H^1(L,({\bf Z}^L)_L) \longrightarrow H^1(L,({\bf G}^L)_L)].
$$
By Example~\ref{quasi-split}, 
every element in $H^1(L,({\bf Z}^L)_L)$ is $R$-trivial 
and therefore the norm principle holds for each such element.
Thus
$$
cor^L_K(res_L([\xi]))\in {\rm Ker}\,[
H^1(K,{\bf Z}^L)\to H^1(K,{\bf G}^L).
$$ 
The claim follows upon noting that
$$
cor^L_K(res_L([\xi]))=[\xi]^3=[\xi],$$ because $H^1(K,{\bf Z}^L)$ is a group of exponent $2$. 
\end{proof}
In the course of the proof of the proposition we established the following.

\begin{corollary}\label{reduction} Let $F/K$ be an arbitrary field extension. 

\smallskip

\noindent
{\rm (a)} We have  ${\rm Ker}\,[H^1(F,{\bf H})\to H^1(F,{\bf Str}(A)]=1$. 

\smallskip

\noindent
{\rm (b)} The variety ${\bf H}\times U$ is birationally isomorphic to  ${\bf Str}(A)$; in particular
${\bf H}$ is $R$-trivial if and only if so is ${\bf Str}(A)$. 
\end{corollary}
\begin{proof} (a) was already established and (b) follows immediately from (a).
\end{proof} 
\begin{corollary} Let $K$ be a field of characteristic $\not=2,3$. 
Let ${\bf H}$ be an arbitrary algebraic group of type ${\rm F}_4$ over $K$ arising from
the first Tits construction. Then ${\bf H}$ is $R$-trivial.
\end{corollary}
\begin{proof} Let $A$ be the corresponding Albert algebra. By~\cite{Thakur1} and \cite{Thakur2}, 
the structure group 
${\bf Str}(A)$ is $R$-trivial. Hence the assertion follows from the above corollary.
\end{proof} 

\begin{remark}
In a forthcoming paper \cite{ACP}, we will deal with automorphism and structure groups of arbitrary Albert algebras.
\end{remark}


\begin{thebibliography}{1}

\bibitem{Als} S.\ Alsaody, \emph{Albert algebras over rings and related torsors}, arXiv:1901.04459 (2019).

\bibitem{ACP}
S.\ Alsaody, V.\ Chernousov and A.\ Pianzola, \emph{On the Tits--Weiss conjecture and the Kneser--Tits conjecture 
for $E_{7,1}^{78}$ and $E_{8,2}^{78}$}, preprint (2019). 

\bibitem{CM} V.\ Chernousov and A.\ Merkurjev, \emph{$R$-equivalence and special unitary groups}, J.\ Algebra 209, 175--198 (1998).

\bibitem{CM1}
V.\ Chernousov and A.\ Merkurjev, \emph{$R$-equivalence in spinor groups}, J. Amer. Math. Soc., {bf 14} (2001),
3, 509--534.

\bibitem{CTS} J.-L.\ Colliot-Th\'el\`ene and J.-J.\ Sansuc, \emph{La $R$-\'equivalence sur les tores}, Ann.\ Sci.\ \'Ecole Norm.\ Sup.\ (4), 10, 175--229 (1977).

\bibitem{Gille}
P.\ Gille, \emph{La $R$-\'equivalence sur les groupes alg\'ebriques r\'eductifs d\'efinis sur un corps global},
Publications math\'ematiques de l'I.H.E.S., {\bf 86} (1997), 199-235.


\bibitem{Jac}
N.\ Jacobson, \emph{Structure and Representations of Jordan Algebras}, AMS Colloquium Publications, Vol.\ 39, 1968.

\bibitem{KMRT} M.-A.\ Knus, A.\ Merkurjev, M.\ Rost, J.-P.\ Tignol, \emph{The book of involutions}, AMS Colloquium Publications, Vol.\ 44, 1998.

\bibitem{M}
Yu.\ I.\ Manin, \emph{Cubic forms}, North-Holland, Amsterdam, 1974.


\bibitem{Merkurjev1}
A.\ Merkurjev, \emph{$R$-equivalence and rationality problem for semisimple adjoint classical 
algebraic groups}, Publications math\'ematiques de l'I.H.{\'E}.S., {\bf 84} (1996),  189-213.

\bibitem{Mil}
J.\ Milnor, \emph{Introduction to algebraic K-theory}, Ann. Math. Stud., N 72,
Princeton, 1971.



\bibitem{P}
H.\ P\. Petersson, \emph{A survey on Albert algebras}, Transformation Groups, {\bf 24} (2019), 219--278.

\bibitem{Sch} R.\ D.\ Schafer, \emph{The exceptional simple Jordan algebras}, Amer.\ J.\ Math.\ {\bf 70} (1948), 82--94.

\bibitem{SV} T.\ A.\ Springer and F.\ D.\ Veldkamp, \emph{Octonion algebras, Jordan algebras and exceptional groups}, Springer Monographs in Mathematics, Springer-Verlag (2000).


\bibitem{Thakur1}
M.\ Thakur, \emph{Automorphisms of Albert algebras and a
conjecture of Tits and Weiss}, Trans.\ AMS {\bf 365} (2013), 3041--3068.
  

\bibitem{Thakur2}
M.\ Thakur, \emph{Automorphisms of Albert algebras and a conjecture of Tits and Weiss II},
Trans.\ AMS {\bf 372} (2019), 4701--4728.


\bibitem{T66}
J.\ Tits, \emph{Classification of algebraic semisimple groups},
In Algebraic Groups and discontinuous Subgroups,
(eds. A.Borel and G.D.Mostow), Proc. Symp. Pure Math. {\bf 9} (1966),
33--62.


\bibitem{V}
V.\ E.\ Voskresenskii, \emph{Algebraic tori}, Nauka, Moscow, 1977.



\end{thebibliography}
\end{document}